\documentclass{article}

\usepackage{amssymb,amsthm,amsmath,marvosym,amsfonts,fontenc}
\usepackage[english]{babel}

\usepackage{mathtools}
\usepackage{xcolor}
\usepackage{hyperref}

\usepackage{tikz}
\usetikzlibrary{arrows.meta,positioning,shapes.geometric}

\usepackage{authblk}
\usepackage{comment}

\hypersetup{
    colorlinks=true,
    linkcolor=blue!50!black,
    citecolor=blue!50!black,
    urlcolor=blue!50!black
}

\newtheorem{theorem}{Theorem}[section]

\newtheorem{conjecture}[theorem]{Conjecture}
\newtheorem{corollary}[theorem]{Corollary}

\theoremstyle{definition}

\theoremstyle{remark}

\title{Odd minors or odd immersions in graphs with independence number two}

\author[1]{Antonia Berm\'udez} 
\author[2]{Bruno L. Netto}
\author[1]{Daniel A. Quiroz}

\date{}
	
\affil[1]{\small Instituto de Ingenier\'ia Matem\'atica-CIMFAV, Universidad de Valpara\'iso, Chile.}

\affil[2]{\small Universidade Federal do Rio de Janeiro, Rio de Janeiro, Brazil.}

\begin{document}

\maketitle

\begin{abstract}

K\"{u}hn, Sauermann, Steiner and Wigderson recently disproved the Odd Hadwiger Conjecture, even for graphs
with independence number $2$. For this class of graphs the conjecture is known to be equivalent to the following:
every $n$-vertex graph $G$ with independence number $2$ contains $K_{\lceil \frac n2 \rceil}$ as an odd minor. 
While this does not hold, we prove that every graph $G$ with independence number $2$ contains $K_{\lceil \frac n2 \rceil}$
as an odd minor or as a totally odd immersion. 

\end{abstract}

\section{Introduction}

Hadwiger's conjecture~\cite{hadwiger1943klassifikation} asserts that every (loopless)
graph contains $K_{\chi(G)}$ as a minor. This aims at widely generalising of the
Four Colour Theorem, which has been shown to imply cases $\chi(G)=5$ and $\chi(G)=6$ \cite{Had6}. 
This celebrated conjecture remains open whenever $\chi(G) \geq 7$. It even remains open for the innocent-looking class of
graphs with independence number $2$. We denote  the independence number of a graph $G$ by $\alpha (G)$. Plummer, Stiebitz
and Toft~\cite{PlummerST} proved that for this class of graphs the conjecture  is equivalent to showing that every $n$-vertex
graph with $\alpha(G)=2$ contains $K_{\lceil \frac{n}{2} \rceil}$ as a minor. A result of Duchet and Meyniel~\cite{DuchetMeyniel}
implies that every graph with $\alpha(G)=2$ contains $K_{\lceil \frac{n}{3} \rceil}$ as a minor. Despite much related work 
(see e.g.~\cite{Blasiak, chudnovsky2012seagulls, Fox2010, KawaSong2007, NorinSeymour}), it is still open whether there is a 
constant $c>\frac 13$ such that every $n$-vertex graph~$G$ with $\alpha(G)=2$ contains $K_{cn}$ as a minor. 

Gerards and Seymour (see~\cite{gerards1995odd}) proposed a strengthening of Hadwiger's conjecture which is known as the Odd Hadwiger Conjecture and states that every graph $G$ contains $K_{\chi(G)}$ as an odd minor.
This strengthening has received considerable attention (e.g.~\cite{GeelenGRSV, KawaSongOdd,NorinSong}), the case $\chi(G)=4$ following from a result of Catlin~\cite{Catlin}. Steiner~\cite{steiner2022asymptotic} proved that the asymptotics
related to the Odd Hadwiger Conjecture are tied to those of Hadwiger's conjecture up to a factor of $2$. And yet, K\"uhn, Sauermann, Steiner and Wigderson, recently 
 disproved the Odd Hadwiger Conjecture~\cite{kuehn2025disproof} by constructing graphs which do not 
contain $K_t$ as an odd minor but have chromatic number at least $(\frac{3}{2} - o(1))t$. The counterexamples given by these authors are precisely graphs with independence number 2, and their result is tight for this class of graphs. 

Shortly before these counterexamples appeared, Ji, Song, Weiss and Zhang~\cite{jia2025oddclique} followed arguments of Plummer, Stiebitz and Toft to prove the following.

\begin{theorem}[Ji, Song, Weiss and Zhang~\cite{jia2025oddclique}]\label{thm:equivminor}
    Every  $n$-vertex graph $G$ with $\alpha(G)\le 2$ contains $K_{\lceil \frac n2 \rceil}$ as an odd minor if, and only if, every such graph contains $K_{\chi(G)}$ as an odd minor.
\end{theorem}

In 1982, Duchet and Meyniel~\cite{DuchetMeyniel} conjectured a weakening of Hadwiger's conjecture, namely, that every $n$-vertex graph $G$ contains $K_{\lceil \frac n{\alpha(G)}\rceil}$ as a minor. While their conjecture remains open, Theorem~\ref{thm:equivminor} and the above mentioned counterexamples imply that even the corresponding weakening of the Odd Hadwiger Conjecture does not hold.

\begin{corollary}\label{coro:coro}
    There exist $n$-vertex graphs $G$ with $\alpha(G)=2$ that do not contain $K_{\lceil \frac n2 \rceil}$ as an odd minor. 
\end{corollary}

A conjecture analogous to Hadwiger's, is that of Lescure and Meyniel ~\cite{LescureMeyniel}, which asks whether every graph $G$ contains $K_{\chi(G)}$ as an immersion.
A graph~$G$ contains another graph $H$ as an \underline{immersion} if there exists an one-to-one mapping $\varphi\colon V(H)\rightarrow V(G)$ such that:
\begin{itemize}
\item[(I)] For every $uv\in E(H)$, there is a path in $G$, denoted $P_{uv}$, with endpoints $\varphi(u)$ and~$\varphi(v)$.
\item[(II)] The paths in $\{P_{uv} \mid uv\in E(H) \}$ are pairwise edge disjoint.
\item[(III)] The vertices of $\varphi(V(H))$, called \underline{terminals}, do not appear as internal vertices on paths~$P_{uv}$.
 \end{itemize}
Note that if condition (II) is strengthened so that the paths are internally vertex-disjoint, then $G$ contains~$H$ as a subdivision.  Thus, if $G$ contains $H$ as a subdivision, 
then it contains $H$ as an immersion (and as a minor). If conditions (I) and (II) are fulfilled but not necessarily (III), then we say that $G$ contains $H$ as a \underline{weak immersion}. 
If $G$ contains $H$ as an immersion where all the corresponding paths are of odd length, then we say that $G$ contains $H$ as a \underline{totally odd immersion}.

The conjecture of Lescure and Meyniel is known to hold  for $\chi(G)\le 7$ thanks to results of Lescure and Meyniel~\cite{LescureMeyniel} and of  DeVos, Kawarabayashi, Mohar and 
Okamura~\cite{DeVosKMO}. It remains open for higher values of $\chi(G)$. For graphs with independence number 2, Vergara~\cite{vergara2017complete} showed that for this class of 
graphs the conjecture  is equivalent to showing that every $n$-vertex graph with $\alpha(G)=2$ contains $K_{\lceil \frac{n}{2} \rceil}$ as an immersion. She proved that every such graph 
contains $K_{\lceil \frac{n}{3} \rceil}$ as an immersion, while Gauthier, Le and Wollan~\cite{GauthierLW} improved this to $K_{2\lfloor \frac{n}{5} \rfloor}$. Moreover, Botler et 
al.~\cite{botler2025largecliques} proved that every graph~$G$ with $\alpha(G)=2$ that meets certain restrictions on the maximum degree contains an immersion of $K_{\lceil \frac{n}{2} \rceil}$, 
and Dahlke~\cite{Dahlke} showed that every graph $G$ with $\alpha(G)=2$ contains a weak immersion of $K_{\lceil \frac{n}{2} \rceil}$.

The following conjecture, which is inspired by a weaker conjecture of Churchley~\cite{Churchley}, is the immersion-analogue of the Odd Hadwiger Conjecture.

\begin{conjecture}[Jim\'enez, Quiroz and Thraves Caro~\cite{JimenezQT}]\label{conj:tosi}
Every graph $G$ contains $K_{\chi(G)}$ as a totally odd immersion.
\end{conjecture}

Conjecture~\ref{conj:tosi} holds whenever $\chi(G)\le 4$ due to a result independently proved by Thomassen~\cite{Thomassen} and Zang~\cite{Zang}. It is open for larger values of $\chi(G)$ and also for graphs with independence number 2. The result of Gauthier, 
Le and Wollan mentioned above holds for totally odd immersions as well, that is, every $n$-vertex graph with $\alpha(G)=2$ contains a totally odd immersion of 
$K_{2\lfloor \frac{n}{5} \rfloor}$. Additionally, every graph with $\alpha(G)=2$ and no totally odd immersion of $K_t$ satisfies $\chi(G)\le \frac{3(t-1)}{2}$, 
as proved by Echeverr\'ia and McDonald~\cite{HenryJessica}.

In this note we prove the following, which contrasts with Corollary~\ref{coro:coro}.

\begin{theorem}\label{thm:main}
Every graph $G$ with $\alpha(G)=2$ contains $K_{\lceil \frac n2 \rceil}$ as an odd minor or as a totally odd immersion.
\end{theorem}

Our approach to Theorem~\ref{thm:main} is inspired by the following result of Chudnovsky and Seymour~\cite{chudnovsky2012seagulls}.

\begin{theorem}[Chudnovsky and Seymour~\cite{chudnovsky2012seagulls}]\label{thm:cs}
Let $G$ be an $n$-vertex graph  with $\alpha(G)=2$ and 
    \[
    \omega(G)\ge
    \begin{cases}
         \lceil \frac n4\rceil , & \text{if $n$ is even},\\
        \lceil \frac {n+3}4\rceil , & \text{if $n$ is odd}.
    \end{cases} 
    \]
    Then $G$ contains $K_{\lceil \frac n2 \rceil}$ as a minor.
\end{theorem}

The odd minor strengthening of Theorem~\ref{thm:cs} was recently proved, first for graphs of odd order by Chen and Deng~\cite{chendeng}, and then in general by 
Ji, Song, Weiss and Zhang~\cite{jia2025oddclique}. Thus, in order to prove Theorem~\ref{thm:main}, it suffices to prove the following. 

\begin{theorem}\label{thm:omegaimmersion}
Let $G$ be an $n$-vertex graph  with $\alpha(G)=2$ and 
    \[
    \omega(G)<
    \begin{cases}
         \lceil \frac n4\rceil , & \text{if $n$ is even},\\
        \lceil \frac {n+3}4\rceil , & \text{if $n$ is odd}.
    \end{cases} 
    \]
    Then $G$ contains $K_{\lceil \frac n2 \rceil}$ as a totally odd immersion.
\end{theorem}

In light of Corollary~\ref{coro:coro} and the odd minor strengthening of Theorem~\ref{thm:cs}, we see that there cannot be an analogue of Theorem~\ref{thm:omegaimmersion} for odd minors.

We prove Theorem~\ref{thm:omegaimmersion} in the next section, and make related conjectures in Section~\ref{sec:problems}.

\section{Odd minors or odd immersions}

All graphs in this paper are finite and loopless. For a graph $G$ and $v\in V(G)$, we denote $\bar N(v)=V(G)\setminus (N(v)\cup \{v\})$. If $G$ is a graph with $\alpha(G)=2$, 
then every vertex $v\in G$ satisfies that $\bar N (v)$ is a clique. In particular, if $v$ witnesses the minimum degree of $G$, then we have $$\delta(G)=n-1-|\bar N (v)|\ge n-1-\omega(G).$$ 
We conclude that in order to prove Theorem~\ref{thm:omegaimmersion} it is enough to prove the following.

\begin{theorem}\label{thm:mindegree}
Let $G$ be an $n$-vertex graph  with $\alpha(G)=2$ and 
    \[
    \delta(G)\ge
    \begin{cases}
        n - \lceil \frac n4\rceil , & \text{if $n$ is even},\\
        n- \lceil \frac {n+3}4\rceil , & \text{if $n$ is odd}.
    \end{cases} 
    \]
    Then $G$ contains $K_{\lceil \frac n2 \rceil}$ as a totally odd immersion.
\end{theorem}

Theorem~\ref{thm:mindegree} would easily be implied by Corollary 5.9 of~\cite{Churchley}. But, as far as we can see, the proof of this corollary has at least one serious error. 
However, we are able to use some of its ideas, together with a careful choice of the terminals and a (necessarily) different counting argument to obtain our result.

Before going to the proof, we add one more piece of notation. If $G$ is a graph and $X,Y\subseteq V(G)$ are disjoint, then we denote by $G[X,Y]$ the graph with vertex set  
$X\cup Y$, and with exactly the edges of $G$ with an endpoint in $X$ and the other endpoint in $Y$.

\begin{proof}[Proof of Theorem~\ref{thm:mindegree}]
    We take a set $A\subseteq V(G)$ on $\lceil\frac n2 \rceil$ vertices with as many edges as possible. Set $B=V(G)\setminus A$ and let $M$ be a maximum matching in  
    $\overline{G}[A,B]$. We enumerate the vertices of $A$ and $B$ in the following way. Vertices $a_1,a_2,\dots, a_{|M|}\in A$ and $b_1, b_2,\dots, b_{|M|}\in B$ are such that $a_ib_i\in M$ for every $1\le i\le |M|.$ For their part, $a_{|M|+1},\dots ,a_{\lceil \frac n2 \rceil}$ and $b_{|M|+1},\dots, b_{\lfloor \frac n2 \rfloor}$ are not endpoints of any edge of $M$. Therefore, since M is maximal,
    we have $a_ib_i \in E(G)$ for every $i$ such that $M+1 \le i \le \lfloor \frac n2 \rfloor$.

    By our choice of $A$, for every $1\le i\le \lceil \frac n2 \rceil -1$, we have $$|N(a_i)\cap A|\ge |N(b_i)\cap A\setminus \{a_i\}|\ge |N(b_i)\cap A|- 1.$$ Our 
    first step will be to make one of these inequalities strict, at least when there is some $j>i$ such that $a_ia_j\notin E(G)$.

    Set $M_0=M$ and suppose that we have indices $|M_0|+1\le i<j\le \lceil\frac n2 \rceil$ with $a_ia_j\notin E(G)$ and $|N(a_i)\cap A|=|N(b_i)\cap A|-1.$ We can assume, 
    without loss of generality, that $i=|M_0|+1$. We relabel as follows: $c_i:=a_i$, $a_i:=a_j$, $a_j:=b_i$, and $b_i:=c_i$. The new set $A:=\{a_1,\dots , a_{\lceil \frac n2 \rceil}\}$, 
    has as many edges as the previous one, but now $M_1:=M_0\cup \{a_ib_i\}=M_0\cup \{a_{|M_0|+1}b_{|M_0|+1}\}$ is a matching in $\overline{G}$. If after the relabeling there exists 
    $b_k$ with $|M_1|+1\le k \le \lceil\frac n2 \rceil$ such that $a_jb_k\notin E(G)$, then we swap $a_{|M_1|+2}$ and $a_j$, and also swap $b_{|M_1|+2}$ and $b_k$, and 
    update $M_1:=M_1\cup \{a_{|M_1|+2}b_{|M_1|+2}\}$. This last addition guarantees that $M_1$ is maximum in the new  $\overline{G}[A,B]$.
    
    Now, if there exists a pair of indices $|M_1|+1\le i<j\le \lceil\frac n2 \rceil$ with $a_ia_j\notin E(G)$ and $|N(a_i)\cap A|=|N(b_i)\cap A|-1$, then we can do similar 
    exchanges to get a larger matching (which is maximum) in the updated $\overline{G}[A,B]$. Since we always get a larger matching, the process ends with a matching $M_\ell$ 
    such that if $|M_\ell|+1\le i<j\le \lceil\frac n2 \rceil$ and $a_ia_j\notin E(G)$, then we have $|N(a_i)\cap A|>|N(b_i)\cap A|-1$, that is, $|N(a_i)\cap A| \ge  |N(b_i)\cap A|$. 
    Moreover,  $M_\ell$ is maximum in the final $\overline{G}[A,B]$.

If $1\le i\le |M_\ell|$, then we have $|N(b_i)\cap A|=|N(b_i)\cap A\setminus \{a_i\}|\le |N(a_i)\cap A|$ since we have $a_ib_i\in M_\ell$ and, consequently, $a_ib_i\notin E(G).$ 
So in general, for  $1\le i\le \lceil \frac n2 \rceil$ we have 
    \begin{equation}\label{eq:neighbourbi0}
    \text{if there exists } j>i \text{ such that } a_ia_j\notin E(G), \text{then we have } |N(b_i)\cap A|\le |N(a_i)\cap A|,
    \end{equation}
as we hoped for. 

    Set $A_1=\{a_1,\dots , a_{|M_\ell|}\}$, $B_1=\{b_1,\dots , b_{|M_\ell|}\}$, $A_2=A\setminus A_1$, and $B_2=B\setminus B_1$. For $a_i\in A_2$ and $b_j\in B_2$ we 
    must have $a_ib_j\in E(G)$ as otherwise we could add the edge $a_ib_j$ to $M_\ell$ contradicting its maximality. If we have $a_i\in A_1$, then $a_ib_i\notin E(G)$ and, 
    since $\alpha(G)=2,$ for every $j\ne i$ we have at least one of $a_ja_i\in E(G)$ or $a_jb_i\in E(G).$ Altogether, if we have $1\le i<j\le \lceil \frac n2 \rceil$, 
    then we have at least one of $a_ja_i\in E(G)$ or $a_jb_i\in E(G).$ Therefore, for $i<j$ the following is always an edge of $G$,
    \[
    f_{ij}=
    \begin{cases}
        a_ia_j , & \text{if $a_ia_j\in E(G)$},\\
        b_ia_j , & \text{otherwise}.
    \end{cases} 
    \]
    We also define $F=\cup_{i<j}\{f_{ij}\}$. 
    
    The set $A$ will be our set of terminals, and if for $i<j$ we have $a_ia_j\notin E(G)$, our immersion will have a length-3 path $a_jb_iwa_i$, where the first edge is in 
    $F$ (by definition), and where $w$ is a common neighbour of $a_i$ and $b_i$ in $B$. To see that we can do this and keep the paths edge-disjoint, we need a lower bound on 
    the number of common neighbours that $a_i$ and $b_i$ have in $B$ after $F$ is removed. For this, we first bound the number of neighbours each of $a_i$ and $b_i$ has in $B$. 
    For our purposes, we can restrict ourselves to the case where there is some $j>i$ such that $a_ia_j\notin E(G)$.

    Set 
\[
    r=
    \begin{cases}
         \lceil \frac n4\rceil , & \text{if $n$ is even},\\
         \lceil \frac {n+3}4\rceil , & \text{if $n$ is odd},
    \end{cases} 
    \]
and for every $a_i\in A$, set $$m_{a_i}=|\bar N (a_i)\cap A|.$$ Since we have $\delta(G)\ge n-r$ we obtain 
\begin{equation*}
\begin{split}
|N_G(a_i)\cap B| & \ge \delta(G)-|N_G(a_i)\cap A| \\ 
&\ge n-r - (\lceil \frac n2 \rceil -1 -m_{a_i}) \\
&=\lfloor\frac n2\rfloor -r+m_{a_i}+1.
\end{split}
\end{equation*}
Let us see what happens to this bound when $F$ is removed. For every edge $f_{ik}$ that joins $a_i$ to some vertex in $B$, there is a vertex $a_k$ with $k<i$ such that $a_ia_k\notin E(G)$.  
Thus if there is some $j>i$ such that $a_ia_j\notin E(G)$, then we have 
\begin{equation}\label{eq:neighbourai}
\begin{split}
|N_{G-F}(a_i)\cap B| &\ge|N_{G}(a_i)\cap B|-(m_{a_i}-1) \\
&\ge \lfloor\frac n2\rfloor -r+2.
\end{split}
\end{equation}

There is no edge of $F$ in $G[B]$. Therefore, by using \eqref{eq:neighbourbi0} and \eqref{eq:neighbourai}, we obtain that when there is some $j>i$ such that $a_ia_j\notin E(G)$ we have
\begin{equation}\label{eq:neighbourb_i1}
\begin{split}
|N_{G-F}(b_i)\cap B| & = |N_{G}(b_i)\cap B|\\
& =|N_G(b_i)|- |N_G(b_i)\cap A|\\
& \ge |N_G(b_i)|- |N_G(a_i)\cap A|\\
&\ge n-r - (\lceil \frac n2 \rceil -1 -m_{a_i}) \\
&\ge \lfloor\frac n2\rfloor -r+m_{a_i}+1.
\end{split}
\end{equation}

We can now get a bound on the number of common neighbours that $a_i$ and $b_i$ have in $B$ after $F$ is removed. We have  $|B|\ge |N_{G-F}(a_i)\cap B|+|N_{G-F}(b_i)\cap B|-|N_{G-F}(a_i)\cap N_{G-F}(b_i)\cap B|$. Rearranging and using \eqref{eq:neighbourai},  \eqref{eq:neighbourb_i1}, and the information in Table 1,  we obtain that when there is some $j>i$ such that $a_ia_j\notin E(G)$ we have
\begin{equation}\label{eq:commonneigh}
\begin{split}
|N_{G-F}(a_i)\cap N_{G-F}(b_i)\cap B| & \ge 2(\lfloor\frac n2\rfloor -r)+m_{a_i}+3-\lfloor\frac n2\rfloor\\
&\ge m_{a_i}.
\end{split}
\end{equation}

\begin{table}[h!]
\centering
\begin{tabular}{|c|c|c|}
\hline
$n$ (mod 4) & $\lfloor\frac n2\rfloor -2r$  \\ \hline
1 & $\frac{n-1}2 -\frac{n+3}2 =-2$   \\
2 & $\frac{n}2 -\frac{n+2}2 =-1$   \\
3 & $\frac{n-1}2 -\frac{n+5}2 =-3$   \\
 \hline
\end{tabular}\caption{Values of $\lfloor\frac n2\rfloor -2r$ according to the value of $n$ (mod 4).}

\end{table}

To end the proof, we construct an auxiliary bipartite graph $H$ as follows. We take $V(H)=C\cup D$, where 
$C$ is the set of non edges in $G[A]$, and 
$D$ is the set of edges of $G$ with an endpoint in $A$ and the other endpoint in $B$. For every pair of indices $i<j$, if we have $a_ia_j\in C$, then we put an edge in 
$H$ joining $a_ia_j$ with $a_iw\in D$ whenever we have $w\in N_{G-F}(a_i)\cap N_{G-F}(b_i)\cap B$.

If we find a matching in $H$ that covers $C$, then for every non edge $a_ia_j$ of $G[A]$, say with $i<j$, we can use the path $a_iwb_ia_j$ to obtain a totally odd immersion of 
$K_{\lceil \frac n2 \rceil}$ with $A$ as its set of terminals. (Note that $b_ia_j\in F$, and so it is not used in any other path, while the edges $a_iw$ and $wb_i$ are also not 
used in any other path because we have a matching in $H$.) Thus it suffices to find such a  matching. To use Hall's Theorem we consider an arbitrary $S\subseteq C$. For $a_ia_j\in S$ 
with $i<j$ we have $|N_H(a_ia_j)|=|N_{G-F}(a_i)\cap N_{G-F}(b_i)\cap B|$. In fact, for $i$ fixed we have 
\begin{equation}\label{eq:last}
    |\displaystyle\bigcup_{i<j,\,\, a_ia_j\in S}N_H(a_ia_j)|=|N_{G-F}(a_i)\cap N_{G-F}(b_i)\cap B|.
\end{equation}

Let $T_S$ be the set of indices $i\in \{1, \dots ,\lceil\frac n2\rceil\}$ such that there exists $a_ia_j\in S$ with $i<j.$ Recalling that we have set $m_{a_i}=|\bar N (a_i)\cap A|$, 
and using \eqref{eq:commonneigh} and \eqref{eq:last}, we obtain

\begin{equation*}
\begin{split}
|S| & \le |\cup_{i \in T_s} \bar N(a_i)\cap A|\\
&\le \sum_{i \in T_s} |\bar N(a_i)\cap A|\\
& \le \sum_{i \in T_s} |N_{G-F}(a_i)\cap N_{G-F}(b_i)\cap B|\\
& = \sum_{i \in T_s} |\displaystyle\bigcup_{i<j,\,\, a_ia_j\in S}N_H(a_ia_j)|\\
& = |\biguplus_{i \in T_s} \,\,\displaystyle\bigcup_{i<j,\,\, a_ia_j\in S}N_H(a_ia_j)|\\
& = |N_H(S)|,
\end{split}
\end{equation*}
where we use $\biguplus$ to denote disjoint union. Therefore, by Hall's Theorem, there is in $H$ a matching that covers $C$, and the result follows.
\end{proof}

We note that if instead of a totally odd immersion we only looked for an immersion in Theorem~\ref{thm:mindegree}, then we could use Lemma 2.1 of paper~\cite{FoxWei} of Fox and 
Wei, which gives an immersion where all paths are of length 1 or~2.

\section{Open problems}\label{sec:problems}

In light of Theorem~\ref{thm:main} we pose a conjecture in the spirit of the work of Duchet and Meyniel~\cite{DuchetMeyniel}.

\begin{conjecture}
    Every $n$-vertex graph $G$ contains $K_{\lceil \frac n{\alpha(G)}\rceil}$ as an odd minor or as a totally odd immersion.
\end{conjecture}

Approaching this conjecture from the odd minor side is a result of Kawarabayashi and Song~\cite{KawaSongOdd}, and from the odd immersion side a result of Bustamante, 
Quiroz, Stein and Zamora~\cite{BustamanteQSZ}.

Further, we wonder whether every graph satisfies the Odd Hadwiger Conjecture or Conjecture~\ref{conj:tosi}.

\begin{conjecture}\label{conj:oddorodd}
    Every graph $G$ contains $K_{\chi(G)}$ as an odd minor or as a totally odd immersion.
\end{conjecture}

It would be nice to see, at least, a linear approximation of this result. If we allow for totally odd \emph{weak} immersions, then there is a linear bound due to 
McFarland~\cite{McFarland}. It would also be nice to see a short(ish) proof of the case $\chi(G)=5$, which was announced for Odd Hadwiger by Guenin~\cite{Guenin} about 
20 years ago, and is open for Conjecture~\ref{conj:tosi}.

Conjecture~\ref{conj:oddorodd} is the odd generalisation of a conjecture posed by the third author in the 2024 CODICIS Workshop (\url{https://dquirozb.github.io/daniel/codicis.html}).

\begin{conjecture}[Quiroz]\label{conj:quiroz}
    Every graph $G$ contains $K_{\chi(G)}$ as a minor or as an immersion.
\end{conjecture}

After a breakthrough result of DeVos, Dvo\v{r}\'ak, Fox, McDonald, Mohar, and Scheide~\cite{DeVosDFMMS}, it was proved by Dvo\v{r}\'ak and Yepremyan~\cite{DvovrakYepremyan} that 
every graph $G$ contains  $K_{\lceil \chi(G)/11\rceil}$ as an immersion; this is the closest result to Conjecture~\ref{conj:quiroz}. (An even better bound was proved by Gauthier, 
Le, and Wollan for weak immersions.) The first open case for this conjecture is $\chi(G)=8$. 

\section*{Acknowledgments}

 The third author thanks Henry Echeverr\'ia for helpful comments. Antonia Berm\'udez and Daniel A. Quiroz thankfully acknowledge support from ANID FONDECYT Regular 1252197. 
 Bruno L. Netto and Daniel A. Quiroz thankfully acknowledge support from ANID MATH-AMSUD MATH230035. Bruno L. Netto is supported by CAPES (88887.670803/2022-00).

\bibliographystyle{plain}
\bibliography{references}

\end{document}